\documentclass{amsart}
\usepackage[margin=1.2in]{geometry}
\usepackage[utf8]{inputenc}

\usepackage{amsmath}
\usepackage{amsthm}
\usepackage{amssymb}
\usepackage[matrix,arrow,curve]{xy}
\usepackage{enumitem} 
\usepackage{graphicx}
\usepackage{mdwlist}	
\usepackage{mathrsfs}
\usepackage{tikz}
\usetikzlibrary{patterns}
\usepackage{caption}
\usepackage{subcaption}
\usepackage{textcomp}

\newcommand\setItemnumber[1]{\setcounter{enum\romannumeral\@enumdepth}{\numexpr#1-1\relax}}

 \newcommand{\p}{\mathbb{P}}

\DeclareMathOperator{\Bir}{Bir}
\DeclareMathOperator{\Aut}{Aut}
\DeclareMathOperator{\PGL}{PGL}
\DeclareMathOperator{\GL}{GL}
\DeclareMathOperator{\Cr}{Cr}
\DeclareMathOperator{\SL}{SL}

\DeclareMathOperator{\PSO}{PSO}
\DeclareMathOperator{\Gr}{Gr}
\DeclareMathOperator{\Stab}{Stab}
\DeclareMathOperator{\Pic}{Pic}

\DeclareMathOperator{\A}{\mathfrak{A}}
\renewcommand{\H}{\mathscr{H}}
\newcommand{\C}{\mathbb{C}}
\newcommand{\GQ}{G\mathbb{Q}}
\newcommand{\Q}{\mathbb{Q}}

\DeclareMathOperator{\SG}{\mathfrak{S}}

\newtheorem{theorem}[equation]{Theorem}
\newtheorem{lemma}[equation]{Lemma}
\newtheorem{proposition}[equation]{Proposition}
\newtheorem{corollary}[equation]{Corollary}

\theoremstyle{definition}
\newtheorem{example}[equation]{Example}
\newtheorem{remark}[equation]{Remark}
\makeatletter
\@addtoreset{equation}{section}
\makeatother

\def\ttl {Finite 3-subgroups in Cremona group of rank 3}
\title{\ttl}

\author{Alexandra Kuznetsova} 
\address{
National research university Higher School of economics, Russia, Usacheva str. 6, 119048;
\'Ecole Polytechnique, France, CMLS, Route de Saclay, 91128 Palaiseau.}

 \email{sasha.kuznetsova.57@gmail.com}
\date{}
\begin{document}

\begin{abstract}
 We consider 3-subgroups in groups of birational automorphisms of rationally connected threefolds
 and show that any 3-subgroup can be generated by at most five elements. Moreover, we study groups of regular 
 automorphisms of terminal Fano threefolds and prove that in all cases which are not 
 among several explicitly described exceptions any 3-subgroup of such group can be generated by at most four elements.
\end{abstract}
\maketitle
\section{Introduction}
The group of automorphisms of the projective space of dimension $n$ over the field of complex numbers is isomorphic to the group $\PGL_{n+1}( \mathbb{C})$. 
However, if we study the group of birational automorphisms instead of regular ones, the description of the group is much more complicated.
The group of birational automorphisms of the projective $n$-dimensional space over $\mathbb{C}$  is called \emph{Cremona group of rank $n$};
we denote it $\Cr_n(\mathbb{C})$. In case where $n=1$ this group is isomorphic to $\PGL_2(\mathbb{C})$. 
For larger~$n$ the group is vast. Only for $n=2$ we can describe the group completely; for $n=3$ and higher we know much less.

A possible way to study a very big group is by considering its finite subgroups. Dolgachev and Is\-kov\-skikh~\cite{Dolgachev_Iskovskikh} have 
classified all finite subgroups of $\Cr_2(\mathbb{C})$.  Beauville \cite{Beauville} studied finite abelian~\mbox{$p$-sub}\-groups
of this group (i.e. subgroups of order $p^k$ where $p$ is a prime number). More precisely, he studied embeddings of groups $(\mathbb{Z}/p\mathbb{Z})^r$ to the Cremona group
of rank $2$, so-called elementary $p$-subgroups. We can see that the rank of an arbitrary abelian $p$-subgroup in the Cremona group is not
greater than the maximal rank of the elementary subgroup. The result from Beauville’s paper implies the following theorem:
 \begin{theorem}\label{thm_beau}
Any abelian $p$-subgroup $G$ in $\Cr_2(\mathbb{C})$ can be generated by $r$ elements, where
 \begin{equation*}
  r\leqslant \left\{ \begin{aligned}
   &4, \text{ if } p=2;\\
   &3, \text{ if } p=3;\\
   &2, \text{ if } p\geqslant 5;\\
 \end{aligned} \right. 
 \end{equation*}
 and this bound is sharp.
\end{theorem}
Prokhorov in \cite{2-subgroups} and \cite{Prokh} has proved some restrictions on ranks of abelian $p$-subgroups of~$\Cr_3(\mathbb{C})$
(in those papers Prokhorov considered elementary $p$-subgroups; however, as we saw above these results imply the same results for all abelian $p$-subgroups).
His results hold not only for the group $\Cr_3(\mathbb{C})$, but for a birational automorphism group $\Bir(X)$ of any rationally connected threefold; 
note that in dimension 2 rational connectedness implies rationality.

\begin{theorem}[{\cite[Theorem 1.2]{Prokh}, \cite[Theorem 1.2]{2-subgroups}}]\label{thm_prokh}
 Consider a projective rationally connected complex threefold  and an abelian $p$-subgroup $G$ in~$\Bir(X)$.
 Then $G$ can be generated by~$r$ elements, where
 \begin{equation*}
  r\leqslant \left\{ \begin{aligned}
   &6 \text{, if } p=2 \text{ and this bound is sharp};\\
   &5 \text{, if } p=3;\\
   &4 \text{, if } p=5,\ 7,\ 11 \text{ or } 13;\\
   &3 \text{, if } p\geqslant 17 \text{ and this bound is sharp}.\\
 \end{aligned} \right.
 \end{equation*} 
\end{theorem}

It is natural to ask how much of this remains to be true for not necessarily abelian $p$-subgroups. 
In the paper by Prokhorov and Shramov \cite{p-subgroups} all $p$-subgroups
of~$\Cr_3(\mathbb{C})$ for large prime numbers $p$ were described.
\begin{theorem}[{\cite[Theorem 1.5]{p-subgroups}}]\label{thm_ps}
 Assume that $X$ is a rationally connected complex threefold and~$G$ is a $p$-subgroup of $\Bir(X)$.
 If $p\geqslant 17$, then $G$ is abelian and it can be generated by $3$ elements or less.
\end{theorem}
Our goal is to give a description of all $3$-subgroups of Cremona groups of rank $2$ and $3$.
Our main result is the following:
\begin{theorem}\label{main_theorem_dim3}
 Consider a projective rationally connected complex threefold $X$ and a~\mbox{$3$-sub}\-group $G$ of $\Bir(X)$. 
 Then the following is true:
  \begin{enumerate}
   \item[$1.$]  The group $G$ can be generated by $5$ elements or less.
   \item[$2.$]  If $G$ acts by regular automorphisms on a minimal terminal $G$-Mori fiber space $X_0$ of dimension~$3$
   (for example, on a terminal $\GQ$-Fano variety with Picard number one), then $G$ can be generated 
   by $4$ elements or less in all cases which are not among the following:
  \begin{enumerate}
   \item[$(a)$]  $X_0$ is a Fano variety, the number of its non-Gorenstein singularities 
   equals $9$ and all of them are cyclic quotient singularities of type $\frac{1}{2}(1,1,1)$.
   \item[$(b)$]  $X_0$ is a Gorenstein Fano variety of Picard number one, of genus $7$ or $10$
   and the number of singularities of $X_0$ equals $9$ or $18$.
  \end{enumerate}
  \end{enumerate}
\end{theorem}
The first assertion of this theorem can be deduced from Theorem~\ref{thm_prokh} (see Remark \ref{rmk_proof_main_thm}).
However to prove the second assertion, more work is necessary: possibly this assertion could be used for further improvement  
of the bound on the number of generators of~\mbox{$3$-sub}\-groups
in groups of birational automorphisms of rationally connected complex threefolds (see Remark \ref{rmk_final}).

Remark \ref{rmk_proof_main_thm} and  Theorem \ref{thm_beau} also imply the following theorem.
It can be deduced from the Dolgachev and Iskovskikh classification of finite subgroups of~\mbox{$\Cr_2(\mathbb{C})$,} but we give another proof.
\begin{theorem}\label{main_theorem_dim2}
 Any $3$-subgroup $G$ of $\Cr_2(\mathbb{C})$ can be generated by $3$ elements or less and this bound is sharp.
\end{theorem}

In order to prove Theorem \ref{main_theorem_dim3} we can not repeat the argument used by Prokhorov and Shramov in~\cite{p-subgroups} to estimate the number of generators 
of $p$-subgroups. They were looking for a point on the threefold fixed by the action of $G$ and study its
stabilizer. However in the case $p=3$ it is possible that there are no fixed points. In Example \ref{no_fixed_points_ab} we describe an action 
of the $3$-group~\mbox{$\mathbb{Z}/3\mathbb{Z}\times\mathbb{Z}/3\mathbb{Z}$} on~$\p^2$ with no fixed points. 
Moreover, in contrast to the result in Theorem \ref{thm_ps}, the groups of birational automorphisms of projective plane and rationally connected threefolds contain non-abelian
$3$-subgroups. In particular, the Heisenberg group acts  on $\p^2$. This group and its embedding to $\PGL_3(\mathbb{C})$ are described 
in Lemma \ref{lemma_image_of_H}. Also the proof of Theorem \ref{main_theorem_dim3} uses recent results on the number of Gorenstein
terminal singularities on Fano threefolds \cite{GFano} and \cite{sing_X8912}. In addition, we use new results about
groups of automorphisms of smooth Fano threefolds \cite{Hilb_Fano}.

In Example \ref{ex_dim=3} we show that the bound on the number of generators of 3-groups in case of threefolds can not be smaller than~4.
Thus, our bound is close to being sharp. Possibly the result could be improved if the study would be restricted only to birational
automorphisms of the projective space instead of an arbitrary rationally connected threefold.

The work is organized as follows: in \S2 we list some assertions useful for study of the regular action of~\mbox{3-groups} on varieties.
In \S3 we prove Theorem \ref{main_theorem_dim2} and give some information about action of~\mbox{$3$-groups} on some specific
surfaces. In \S4 we study 3-subgroups in groups of regular automorphisms of Fano threefolds and prove Theorem \ref{main_theorem_dim3}.
In Appendix \ref{appendix} we estimate the number of generators of any 3-subgroup of $\GL_n(\mathbb{C})$ for $n<9$ and 
state several useful assertions from linear algebra.

All varieties here are supposed to be projective, normal and defined over the field $\C$. We denote by~$\SG_n$ the group of permutations of $n$
elements; by $\A_n$ the group of even permutations of $n$ elements; and by~$\mathrm{C}_n$ the cyclic group of $n$ elements.
We denote by $\rho(X)$ the Picard number of $X$, namely, the rank of Picard group $\Pic(X)$.

\smallskip
\textbf{Acknowledgements.}
I am very grateful to my advisor, Constantin Shramov, for suggesting this problem as well as for his patience and invaluable support.
I also thank Artem Avilov, Andrey Trepalin and Yuri Prokhorov for useful discussions. 
This work is supported by  Russian Science Foundation under grant~\mbox{\textnumero{18-11-00121}}.

\section{Elementary properties of $3$-groups}

In this section we state some properties of $3$-groups with a faithful action on algebraic 
varieties and give examples of such actions. We use the results and notations from Appendix \ref{appendix}.

The variety is called \emph{$G$-variety}, if the group $G$ acts faithfully on it. The fiber space $X \to B$
is called~\mbox{\emph{$G$-fiber space}}, if $X$ is a $G$-variety and the action of $G$ preserves the fiber space structure.
For a variety $X$ and a point $x$ on it by $T_x$ we denote the Zariski tangent space to $X$ in $x$.

\subsection{Properties of group actions on varieties}

Consider a finite group $G$ and a $G$-fiber space $X$ over base $B$. Assume that the fiber of $X$ over a general point $B$ is 
isomorphic to $F$. Then $G$ is a following extension of groups:
\begin{equation}\label{base-fiber_es}
 1\to G_F \to G \to G_B \to 1.
\end{equation}
Here $G_F \subset \Aut(F)$ is a subgroup of all elements of $G$ inducing trivial automorphisms of the base of the fiber space 
and~\mbox{$G_B\subset \Aut(B)$} is a quotient group $G/G_F$. 
 \begin{lemma}\label{lemma_Gbundle}
 Assume that $X\to B$ is a $G$-fiber space with a general fiber $F$, and $G$ is a~\mbox{$3$-group}.
 If any~\mbox{$3$-sub}\-group in $\Aut(F)$ and in $\Aut(B)$ can be generated by $n$ or $m$ elements respectively, then $G$ can be generated
 by $n+m$ elements.
 
 In particular, if $X\to B$ is a double cover, then $G\subset \Aut(B)$.
\end{lemma}
\begin{proof}
 Since $G$ is finite and preserves the fiber space structure, we have an exact sequence \eqref{base-fiber_es}. 
 By Lemma \ref{generators_of_group_extension} since groups $G_F$ and $G_B$ can be generated by $n$ and $m$ elements respectively
 $G$ can be generated by $n+m$ elements.
 
 If $X\to B$ is a finite $n$-cover , then $\Aut(F)\cong\SG_n$. In particular, if $n=2$, then  $G_F\subset \mathrm{C}_2$.
 Since $G$ and $G_F$ are both 3-groups, then  $G_F=1$ and~\mbox{$G\subset \Aut(B)$}.
\end{proof}
Next assertion is useful for actions of finite groups which fix a point on the variety.
\begin{proposition}[{\cite[Lemma 4]{Tangent_space}}]\label{fixed_point}
Assume that a finite group $G$ acts on a variety $X$ and fixes a point~$x$. Then $G\subset\GL(T_x)$.
\end{proposition}
Using the next assertion we can construct for any variety $X$ with a birational action of a finite group $G$
its birational model with a regular action of $G$.
\begin{proposition}[{see, for instance, \cite[Lemma-Definition 3.1]{Regularisation}}] \label{regularization}
  Assume that $X$ is a variety and $G$ is a finite subgroup in $\Bir(X)$. Then there exists a smooth $G$-variety $\widetilde{X}$
  and a $G$-equivariant birational map $\widetilde{X} \dashrightarrow X$.
\end{proposition}

The next assertion describes the group of automorphisms of  Grassmannian.
\begin{proposition}[{\cite[Theorem I]{Aut_Gr}}]\label{GR}
 Consider complex vector space $V$ of dimension $n$. If~\mbox{$n \ne 2k$}, we have~\mbox{$\Aut(\Gr(k, V)) \cong\PGL(V)$.} 
 In case~$n=2k$ the group $\PGL(V)$ is a subgroup of $\Aut(\Gr(k, V))$ of index $2$.
\end{proposition}

\subsection{Number of generators of $3$-groups}
We describe an example of the action of a $3$-group on a surface which does not 
fix any point. This is a reason why we can not use the same strategy as in~\cite{p-subgroups}.
Here by $\H_3$ we denote the Heisenberg group, it is described in details in Appendix \ref{appendix}

 \begin{example}\label{no_fixed_points_ab}
 Consider the standard action of the Heisenberg group $\H_3$ on the $3$-dimensional vector space $V$.
 This action induces the faithful action of the group $\mathrm{C}_3\times\mathrm{C}_3$ on the projective 
 plane $\p(V)\cong\p^2$. If $x$ is fixed by this action, then the corresponding linear subspace $W$ of $V$
 is invariant under the action of $\H_3$. However, this contradicts to the fact that the action 
 of $\H_3$ on $V$ induces an irreducible representation.
 \end{example}
The next example shows that the bound in Theorem \ref{main_theorem_dim2} is sharp.
 \begin{example}\label{ex_dim=2}
 Fix a projectivization $\p^3\cong \p(V)$ of the vector space $V$ with coordinates $x_1,\ x_2,\ x_3$ and~$x_4$.
 Consider the following action on $\p(V)$ of the abelian $3$-group~$(\mathrm{C}_3)^{3}$ with generators $\gamma_1,\ \gamma_2$ and~$\gamma_3$.
 Each $\gamma_i$ acts non-trivially only on $x_i$ by multiplication by the primitive cube root of unity;
 and $\gamma_i\cdot x_j = x_j$ for all $i\ne j$. This action is 
 faithful and it preserves the Fermat cubic surface $S$ in $\p^3$:
 \begin{equation*}
  S=\{(x_1:x_2:x_3:x_4)\in\p^3|\ x_1^3+x_2^3+x_3^3+x_4^3 =0  \}.
 \end{equation*}
 Thus, the $3$-group $\mathrm{C}_3^{3}$ acts faithfully
 on a rational surface $S$, though it can not be generated by less than $3$ elements. Therefore, it is a subgroup of the group $\Bir(\p^2)$.
\end{example}
The following example describes the action of a $3$-group that can not be generated by less than $4$ elements on a rationally connected threefold.
\begin{example}\label{ex_dim=3}
 Consider a product $X=\p^1\times S$ of $\p^1$ and a Fermat cubic surface. It is a rational threefold with a faithful
 action of a $3$-group~$(\mathrm{C}_3)^4 \cong (\mathrm{C}_3)^3\times \mathrm{C}_3$. The group  $(\mathrm{C}_3)^3$ acts
 on $S$ as in Example \ref{ex_dim=2}; the action of the group $\mathrm{C}_3$ is standard. Thus, the 
 group $(\mathrm{C}_3)^4$ is a subgroup in~$\Cr_3(\mathbb{C})$, though it can not be generated by less than 4 elements.
\end{example}

\section{Preliminaries}
In this section we prove Theorem \ref{main_theorem_dim2} and give assertions that happen to 
be useful in study of birational automorphisms of rationally connected threefolds. In particular,
we estimate the number of generators of $3$-subgroups in groups of automorphisms of several
concrete curves and surfaces.

\subsection{Proof of Theorem \ref{main_theorem_dim2}}
We start with the assertion connecting ranks of abelian~\mbox{$p$-sub}\-group in a group of birational
automorphisms of a variety and numbers of generators of a non-abelian $p$-subgroup there.
\begin{proposition}\label{prop_el_to_all}
 Assume that for all rationally connected varieties $Y$ of dimension $n$ and any abelian $p$-subgroup
 $A\subset \Bir(Y)$ we know that the rank of $A$ is less than or equal to $r$. If $X$ is rationally
 connected and of dimension $n$, then any $p$-subgroup $G\subset \Bir(X)$ (in particular, non-abelian) can 
 be generated by $r$ elements.
\end{proposition}
\begin{proof}
 By Proposition \ref{regularization} there exists a birational model  $\widetilde{X}$  of $X$ with a regular action of $G$.
 Consider a Frattini subgroup $\Phi(G)\subset G$ the intersection of all maximal subgroups of $G$. The quotient group~\mbox{$A=G/\Phi(G)$}
 is an abelian group and it acts on a quotient variety $Y = \widetilde{X}/\Phi(G)$. This quotient variety is rationally connected and 
 of dimension $n$; thus, the rank of the group $A$ is less than or equal to~$r$. Therefore, by the Burnside 
 theorem \cite[Theorem 12.2.1]{Burnside_thm} the group $G$ can be generated by~$r$ elements.
\end{proof}
Now conclude from this fact Theorem \ref{main_theorem_dim2}.
\begin{proof}[Proof of Theorem \ref{main_theorem_dim2}]
Theorem \ref{thm_beau} gives us the bound on the rank of abelian $3$-groups acting on rational surfaces.
Proposition \ref{prop_el_to_all} implies the result.
\end{proof}
\begin{remark}\label{rmk_proof_main_thm}
 Theorem \ref{thm_prokh} gives us the bound on ranks of abelian $3$-subgroups acting on rationally connected threefolds. 
 Thus, Proposition \ref{prop_el_to_all} implies that any $3$-subgroup $G\subset \Bir(X)$ can be generated by $5$ elements.
\end{remark}
 
\subsection{Curves and surfaces}
In this work we need to estimate $p$-subgroups of the group of automorphisms of a curve $C$ of the genus $g$ greater than 1.
The Riemann–Hurwitz formula shows that the order of such a group is not greater than $84(g-1)$. However, this bound is not
sufficient in some cases. Since all $p$-groups are nilpotent; the result \cite[Theorem 1.2]{Aut_of_curves} implies a better bound for the
order of a $p$-subgroup $G\subset\Aut(C)$:
\begin{equation} \label{eq_estimation_on_curve}
 |G|\leqslant 16(g-1).   
\end{equation}
Using this formula we can study $3$-groups acting on the Jacobian of a curve.
\begin{lemma}\label{lemma_jac}
 Assume that $\mathrm{Jac}(C)$ is the Jacobian of a curve $C$ of genus $2$ and $G$ is a $3$-group with a faithful action
 on $\mathrm{Jac}(C)$ preserving its polarization. Then $G$ can be generated by $2$ elements.
\end{lemma} 
\begin{proof}
 By Torelli theorem for curves automorphisms of $\mathrm{Jac}(C)$ preserving its polarization 
 form a subgroup of the group $\mathbb{Z}/2\mathbb{Z}\times \Aut(C)$. Thus, the group $G$ is a subgroup
 of~$\Aut(C)$. By~\eqref{eq_estimation_on_curve} we get
 \begin{equation*}
    |G|\leqslant|\Aut(C)|\leqslant 16 (g(C)-1) = 16.
 \end{equation*}
 Since the order $|G|$ is a power of $3$, by Lemma \ref{subgroup_of_small_index} we conclude that $G$ can be 
 generated by two elements.
\end{proof}
Now we use the formula \eqref{eq_estimation_on_curve} for a fiber space over a curve.
\begin{lemma}\label{lemma_proj_over_curve}
 Assume that $S$ is a projectivization of a vector bundle of rank $2$ over a curve $C$ of genus~$3$ and  $G$ is 
 a $3$-subgroup in $\Aut(S)$. Then $G$ can be generated by $4$ elements.
\end{lemma}
\begin{proof}
 Any automorphism of such surface commutes with the projection to $C$. By \eqref{eq_estimation_on_curve} any
 $3$-subgroup of $\Aut(C)$ has an order less than or equal to $32$; thus, it can be generated by 3 elements. 
 The fiber of the projection to $C$ is isomorphic to a projective line; thus, by Lemma \ref{gl2} any $3$-group
 with a faithful action on fiber is cyclic. By Lemma \ref{lemma_Gbundle} this implies that $G$ can be generated
 by 4 elements.
\end{proof}

\section{Automorphisms of $3$-dimensional Mori fiber spaces}
\subsection{Preliminaries}
We call $X$  a \emph{$\GQ$-factorial variety}, if it is a $G$ variety and any $G$-invariant Weil divisor is a $\Q$-Cartier divisor,
i.e. its power is a Cartier divisor. The terminal $G$-variety $X$ is a \emph{$\GQ$-Fano variety}, if it is $\GQ$-factorial , the anticanonical
class $-K_X$ is ample and the rank of the group $\Pic(X)^G$ equals 1. 
  
We study groups of regular automorphisms of minimal terminal $G$-Mori fiber spaces of dimension $3$. 
They are either $\GQ$ Fano threefolds or $G$-Mori fiber spaces~\mbox{$X\to B$} such that $\dim(B)>0$, the
class $-K_{X/B}$ is relatively ample and the rank of the relative $G$-invariant Picard group~$\Pic(X/B)^G$ equals 1.

In the following assertion we are describe $3$-subgroups with a regular action on a $G$-Mori fiber space.
\begin{lemma}\label{lemma_mb}
  If $X\to B$ is a $G$-Mori fiber space of dimension $3$ and $\dim(B)>0$, then a $3$-sub\-group~$G$ in $\Aut(X)$ 
  can be generated by $4$ elements.
\end{lemma}
\begin{proof}
   The base $B$ is rationally connected as well as the general fiber of the map to $B$. Thus, by Lemma \ref{lemma_Gbundle} and the bounds
   in Theorem \ref{main_theorem_dim2} and Lamma \ref{gl2} we get the result.
\end{proof}

Our next goal is to study $3$-subgroups in groups of regular automorphisms of $\GQ$-Fano threefolds.
Let us introduce are some important invariants of smooth and Gorenstein Fano varieties.
The \emph{index} of a Gorenstein Fano variety $X$ is a maximal number $r$ such that there exists an element $H$ 
of the Picard group of $X$ such that
 \begin{equation*}
  - K_X \sim rH.
 \end{equation*}
Any automorphism $f$ of $X$ preserves the line bundle~$\mathcal{O}(H)$ since~\mbox{$f^*K_X \sim K_X$} and the group
$\Pic(X)$ is torsion-free by \cite[Proposition 2.1.2]{FANO}. 
 
Another invariant of Fano threefold of index 1 is its \emph{genus}:
\begin{equation*}
  g = \dim(H^0(X, \mathcal{O}(H)))-2.
\end{equation*}
Since $X$ is a Fano variety, the line bundle $\mathcal{O}(H)$ is ample and its linear system induces the following map:
\begin{equation*}
 \phi_{|H|}: X \dashrightarrow \p(H^0(X, \mathcal{O}(H)))^{\vee} \cong \p^{g+1}.
\end{equation*}

We call $X$ a \emph{del Pezzo threefold of degree $d$} if $X$ is a terminal Gorenstein Fano threefold of index $2$ 
and $H^3= d$.
\begin{proposition}[{\cite[Table 12.4]{FANO}}]
Smooth Fano threefold with Picard number $1$ are classified:
 \ 
\begin{enumerate}
   \item[$1.$] If $r=1$, there are ten families of smooth Fano threefolds of genera  $2,\dots,10$ and $12$;
   \item[$2.$] if $r=2$, there are ten families of smooth Fano threefolds; namely, del Pezzo threefolds of degrees $1,\dots,5$;
   \item[$3.$] if $r=3$ there is a unique family of smooth Fano threefolds: namely, smooth quadric hypersurfaces in $\p^4$;
   \item[$4.$] if $r=4$ there is a unique smooth Fano threefold: $\p^3$.
  \end{enumerate}
\end{proposition}
The full classification of smooth Fano threefolds and there invariants (Picard ranks and Hodge numbers $h^{1,2}$) are given in tables.

\subsection{Singular Fano varieties}
 
In dimension three the result of minimal model program is a variety with terminal singularities. 
We need the following assertion.
\begin{theorem}[{\cite[Proposition 3.6]{YPG}}]\label{terminal_singularities}
   Any terminal singularity $x\in X$ of dimension $3$ is locally isomorphic to $Y/\Gamma$, where $\Gamma\cong \mathrm{C}_r$ acts on $Y$ freely
   outside the point $y$. The point $y$ is a $cDV$-singularity or a smooth point (thus, it is a Gorenstein singularity).
   Variety $Y$ in this case is a projectivization of a canonically constructed sheaf of algebras on an open subset of $X$.
\end{theorem}	
  
The next assertion describes groups preserving a terminal point on a threefold $X$.
\begin{lemma}[{c.f. \cite[Lemma 2.4]{Prokh}}]\label{stabilizer_of_singularity}
  Assume that $X$ is a threefold with terminal singularities and~$G$ is a $3$-subgroup in $\Aut(X)$ preserving a point $x\in X$.
  Then $G$ can be generated by $3$ elements.
\end{lemma}
\begin{proof}
Now we assume that $x$ is not a smooth point. Then by Theorem \ref{terminal_singularities} there exists a variety $Y$
and a cyclic group $\Gamma$ such that $X$ is locally isomorphic to a quotient variety $Y/\Gamma$. Since the construction 
of~$Y$ is canonical, an extension $\widetilde{G}$  of $G$ acts on $Y$ and this action commutes with the projection to~$Y/\Gamma$.
Since the Sylow $3$-subgroup of $\widetilde{G}$ maps surjectively to $G$ by Lemma \ref{generators_of_group_extension} it suffices
to prove that $\widetilde{G}$ can be generated by $3$ elements. Thus, we reduct to the case $x$ is $cDV$-singularity.

If $x$ is a Gorenstein terminal point then by \cite[Corollary 3.12]{YPG} the tangent space $T_x$ is of dimension less 
or equal than $4$. Thus, by Proposition \ref{fixed_point} the group $G$ is a subgroup of $\GL_4(\mathbb{C})$ and by 
Proposition \ref{gl4} can be generated by $4$ elements.

Moreover, if $x$ is a smooth point, then $\dim(T_x)=3$ and $G$ can be generated by $3$ elements by Proposition \ref{gl4}.

Assume that $G$ can not be generated by less than $4$ elements. Then by Corollary \ref{corollary_3sgr_gl3_3generators}
it either contains a subgroup of the torus of $\GL_4(\mathbb{C})$ of  rank $4$ or it is a product of a 3-subgroup of a
one-dimensional toris and a group $H$ generated by matrices \eqref{eq_matrix_in_3group}.
  
In the first case replacing elements of $G$ by there degrees we can assume that $G\cong C_3\langle g_1,\dots,g_4\rangle$ is 
isomorphic to $\mathrm{C}_3^4$. 
Choose coordinates $z_1,z_2,z_3$ and $z_4$ of $T_x$ such that the action of $G$ is diagonal.
\begin{equation*}
   g_i\cdot z_j = \left\{ \begin{aligned}
                          \zeta z_i, &\text{ if }i=j;\\
                          z_j, &\text{ otherwise.}
                         \end{aligned}
  \right.
\end{equation*}
Here $\zeta$ is a cube root of unity. Point $x$ is a $cDV$-singularity, it is given by the equation $f(z_1,z_2,z_3,z_4)$
from the list \cite[Section 6]{YPG}. In particular, $f$ is an irreducible polynomial semi-invariant under elements $g_i$.
Moreover, $f$ has a non-trivial quadratic part.

After renumeration of the coordinates we can assume that $z_1^2$ or $z_1z_2$ has non-zero coefficient in~$f$. Then $g_1$ acts on $f$
with a non-trivial eigenvalue. Since on $z_j$ the element $g_1$ acts trivially, all monomials of $f$ are divizible by $z_1$. 
This contradicts to the fact that $f$ is irreducible.

If $G$ contains an element $g_1$ as before and elements $\sigma$ and $t$ preserving $z_1$ and on $z_2,z_3,z_4$ acting as
matrices \eqref{eq_matrix_in_3group}. Consider the quadratic part $f_2$ of $f$. It is semi-invariant under the action of $G$ 
Thus, it is either 0 or $z_1^2$. There is an only equation in the list \cite[Section 6]{YPG} with such a property and in this 
case the cubic part $f_3$ of $f$ is also non-trivial. However, such a polynomial can not be semi-invariant under our group action.
Thus, this situation is also impossible and $G$ can be generated by $3$ elements.
\end{proof}
 
\begin{remark}\label{rmk_stab_cA1}
Lemma \ref{stabilizer_of_singularity} shows that any $cDV$-singularity is not invariant under the faithful action 
of a $3$-group generated by $4$ elements. This result can not be improved: there exist $cDV$-singularities fixed 
by the action of the group $\mathrm{C}_3^3$. For instance, consider $cA_1$-singularity with an equation $z_1^2+z_2z_3+z_4^3$
in a four-dimensional space. It is invariant under the action of a group with following generators:
 \begin{equation*}
  C_3^3\cong \left\langle \mathrm{diag}(\zeta,\zeta^{-1},1,1), \mathrm{diag}(\zeta, 1,\zeta^{-1},1), \mathrm{diag}(1,1,1,\zeta)\right\rangle \subset \GL_4(\mathbb{C})
 \end{equation*}
\end{remark}
Lemma \ref{stabilizer_of_singularity} implies the following corollary
\begin{lemma}\label{lemma_less_27_singularities}
 Assume that $X$ is a terminal Fano threefold and $G$ is a $3$-subgroup of $\Aut(X)$. If the length of the $G$-orbit 
 of a point $x\in X$ equals $N<27$, then $G$ can be generated by $5$ elements. If~\mbox{$N<9$,} then $G$ can be generated by $4$ elements.
\end{lemma}
\begin{proof}
Consider the stabilizer $\Stab(x)$ of $x$. It is a 3-subgroup of $G$ of index $N = 3^k$.
By Lemma \ref{stabilizer_of_singularity} the group $\Stab(x)$ is generated by $3$ its elements.
Lemma \ref{subgroup_of_small_index} implies the result.
\end{proof} 

There is an important fact that any Gorenstein terminal Fano threefold is a deformation of smooth Fano threefold;
it is called Namikawa smoothing \cite{Fano_smoothing}. I.e. there exists a family
 \begin{equation*}
  \mathcal{X} \to B \ni 0,
 \end{equation*}
the general fiber of $\mathcal{X}$ is a smooth Fano threefold and the central $X_0$ is isomorphic to $X$.
Moreover, we can estimate the number $N$ of singularities on the central fiber $X_0$ by parameters of the general
smooth fiber $X_b$:
\begin{equation}\label{number_of_sing}
  N\leqslant 20 - \rho(\mathcal{X}_b) + h^{1,2}(\mathcal{X}_b),
\end{equation}	
here $h^{1,2}(\mathcal{X}_b)$ is a dimension of the cohomology group $H^{1,2}(\mathcal{X}_b)$. 

Namikawa bound implies the following result.
\begin{corollary}\label{corollary_namikawa_imply}
Assume that $X$ is a Gorenstein terminal singular Fano threefold, $\rho(X) = 1$ and~$G$ is a~\mbox{$3$-sub}\-group in $\Aut(X)$.
If for a general fiber in Namikawa smoothing we have~$r(\mathcal{X}_b)=1$ and~\mbox{$g(\mathcal{X}_b)>6$,}
then $G$ can be generated by $5$ its elements. Moreover, if the number of singularities of~$X$ differs from $9$ and $18$
the group~$G$ can be generated by $4$ elements.
\end{corollary}
 \begin{proof}
 Consider a singular point $x$ on $X$. The bound \eqref{number_of_sing} and \cite[Tables 12.6–12.7]{FANO},
 implies that the length of the orbit of $x$ is less than 27. Then by Lemma \ref{lemma_less_27_singularities}
 the group $G$ can be generated by $5$ elements.
 
 If number of of singularities of $X$ differs from $9$ and $18$, then there exists a singularity $x$ on $X$ with
 orbit of less than $9$ points. In this situation by Lemma \ref{lemma_less_27_singularities}
 the group $G$ can be generated by $4$ elements.
\end{proof}

Recall that the \emph{index} of a non-Gorenstein singularity $x$ on a variety $X$ is a minimal number $n$
such that $nK_X$ is a Cartier divisor in a neighborhood $x$.
\begin{corollary}\label{corollary_reid_imply}
Assume that $X$ is a terminal Fano threefold with a non-Gorenstein singularity and~$G$ is a $3$-subgroup
in $\Aut(X)$. Then $G$ can be generated by $5$ elements. Moreover, $G$ can be generated by~$4$ elements 
unless there is exactly $9$ singularities on $X$ and they all are cyclic quotient-singularities of type~$\frac{1}{2}(1,1,1)$.
\end{corollary}
\begin{proof}
For a terminal singularity $x_i$ on $X$ there exists a set (basket) of virtual quotient-singularities~$y_{ij}$ 
of indexes~$r_{ij}$, where $j = 1,\dots,l_i$. By~\mbox{\cite[Corollary 10.3]{YPG}} and \cite{Kaw} we have the following
inequality:
 \begin{equation}\label{eq_RR}
  \sum_{i=1}^N\sum_{j=1}^{l_i} \left( r_{ij} - \frac{1}{r_{ij}} \right)< 24,
 \end{equation}
here $N$ is a number of all non-Gorenstein singularities on $X$. This implies that $N<16$; thus, by 
Lemma \ref{lemma_less_27_singularities} the group $G$ can be generated by $5$ elements.
 
If the number of singularities on $X$ differs from $9$, we can find a non-Gorenstein point $x$ on $X$ 
with orbit of less than 9 points and by 

If $N=9$ and $G$ is transitive on the set of singularities, then numbers~\mbox{$l_i = l$} in formula \eqref{eq_RR} 
coinside and $r_{ij} = r_j$ do not depend on $i$. Thus, we can rewrite the inequality \eqref{eq_RR} in a following way:
 \begin{equation*}
  \sum_{j=1}^l \left(r_j  - \frac{1}{r_j}\right)< \frac{8}{3}.
 \end{equation*}
Then $l=1$ and $r_1=2$ and all non-Gorenstein singularities of $X$ are cyclic quotient-singularities. 
(see the proofs of \cite[Theorem 9.1, Theorem 10.2]{YPG}).
Since $r_1=2$, they are of type $\frac{1}{2}(1,1,1)$.
\end{proof}

\subsection{Fano threefolds with $\rho>1$}
 
In this section we study Fano threefolds with Picard lattice of rank greater than 1.
As before we consider Fano threefolds such that the rank of the group $\Pic(X)^G$ equals 1.
The following assertion is useful in this situation.
 
\begin{theorem}[{\cite[Theorem 1.2, Theorem 6.6]{G-Fano}}]\label{Fano_big_rho}
Assume that $X$ is a terminal Gorenstein Fano threefold, $\rho(X)>1$ and $G$ is a finite group with an action on
$\Pic(X)$ preserving the intersection form and the canonical class. If we have $\Pic(X)^G \cong \mathbb{Z}$, then 
$X$ is one of the following list:
 \begin{enumerate}
  \item[$1.$] $\rho(X) = 2;$
  \item[$2.$] $X\cong\p^1 \times \p^1 \times \p^1;$
  \item[$3.$] There is a double cover $f:X \to\p^1 \times \p^1 \times \p^1$ branched along an element of the linear system~\mbox{$|-K_{\p^1 \times \p^1 \times \p^1}|;$}
  \item[$4.$] $X$ is a blow up of a divisor $W$ of bidegree $(1, 1)$ in $\p^2\times \p^2$ in a curve~$C$ of bidegree $(2,2)$ and 
  the composition $C \hookrightarrow W \to \p^2\times\p^2 \xrightarrow{pr_i} \p^2$ is an embedding for both projections
  $pr_1$ and $pr_2;$
   \item[$5.$] $X$ is a divisor of multidegree $(1,1,1,1)$ in $\p^1 \times \p^1 \times \p^1\times \p^1$.
 \end{enumerate}
 \end{theorem}
This theorem implies the following.
\begin{lemma}\label{Fano_rho>1_smooth}
  Assume that $G$ is $3$-group and $X$ is a terminal Gorenstein $G\mathbb{Q}$-Fano threefold such that~$\rho(X)>1$. 
  Then $G$ can be generated by $4$ elements.
\end{lemma}

\begin{proof}
Consider all cases from Theorem \ref{Fano_big_rho}. 
Assume that $\rho(X)=2$ as in the point 1 of Theorem~\ref{Fano_big_rho}. Then by Lemma \ref{3-group_acts_on_Z^2} 
the sublattice $\Pic(X)^G$ of the rank 2 lattice $\Pic(X)$ can not be of rank 1. Therefore, the threefold $X$ is not $G$-minimal.
  
Consider $3$-subgroup $G$ of the group of automorphisms of  $X=\p^1\times\p^1\times\p^1$ as in the point $2$ of Theorem \ref{Fano_big_rho}).
Consider the subgroup $H$ in $G$ of automorphisms of $X$ with trivial action on $\Pic(X)$. Then $H$ is a subgroup of a 
group~\mbox{$\PGL_2(\mathbb{C})\times \PGL_2(\mathbb{C})\times\PGL_2(\mathbb{C})$}; thus, by Lemma~\ref{gl2} the group $H$ can be generated
by $3$ elements. The quotient group $G/H$ acts faithfully on the lattice $\Pic(X)\cong \mathbb{Z}^3$. Then it is a subgroup 
of $\GL_3(\mathbb{Z})$ and by Lemma \ref{3-subgroups_in_gl3Z} it is cyclic. Therefore, by Lemma \ref{generators_of_group_extension} 
the group $G$ can be generated by $4$ elements.

Consider a $3$-group $G$ with a faithful action on the double cover $f\colon X\to \p^1\times \p^1 \times \p^1$ as in the point 1 of Theorem~\ref{Fano_big_rho}. 
Denote by $H$ the kernel of action of $G$ on $\Pic(X)$. The group $G/H$ acts faithfully on the Picard group $\Pic(X)\cong \mathbb{Z}^3$. Thus, $G/H$ is
cyclic  by Lemma \ref{3-subgroups_in_gl3Z}. The group $H$ fixes all linear bundles on $X$; in particular,
it preserves the linear system $|(pr_i\circ f)^*\mathcal{O}_{\p^1}|$. Thus, any element of $H$ maps the fiber of $f$ to a fiber of $f$
and preserves the structure of the double cover $f$. By Lemmas \ref{lemma_Gbundle} and \ref{gl2} the group $H$ can be generated by
three elements. Therefore, by Lemma \ref{generators_of_group_extension} the group $G$ can be generated by $4$ elements.

Consider the blow up $X$ of $W$ along $C$ as in point $4$ of Theorem \ref{Fano_big_rho}. By~\cite[Table 12.4]{FANO} there exists
three contractions of an exceptional divisor on $X$ to $W$ and there exist three exceptional divisors $E_1$, $E_2$ and $E_3$. 
Consider the normal subgroup $H$ of $G$ which fixes these three divisors. This group $H$ acts regularly on $W$ and fixes there 
the curve $C$. Moreover, since $\rho(W)=2$ the 3-group $H$ acts trivially on $\Pic(W)$. Thus, the group $H$ preserves the structure 
of the projective bundle
\begin{equation*}
 W\cong \p_{\p^2}(T_{\p^2})\to \p^2.
\end{equation*}
Then by Lemma \ref{lemma_Gbundle} the group $H$ can be generated by $3$ elements. 
Therefore, $G$ can be generated by $4$ elements by Lemma \ref{generators_of_group_extension}..
  
Finally, consider the divisor $X$ of multidegree $(1,1,1,1)$ in $\p^1 \times \p^1 \times \p^1 \times \p^1$ 
as in point $5$ of Theorem~\ref{Fano_big_rho}. By the table \cite[Table 12.5]{FANO} there are exactly $4$ different 
contractions from $X$ to~\mbox{$\p^1\times \p^1\times \p^1$}. The $3$-group $G$ can not act transitively on the set 
of $4$ divisors; thus, it fixes one of them. Therefore, $G$ is not a $G$-minimal variety; this contradicts to the assumption.
\end{proof}
 \subsection{Fano threefolds with $\rho=1$}
\subsubsection{$|-K_X|$ does not induce an embedding}
First consider a Fano threefold such that its anticanonical linear system has a non-trivial base locus.
\begin{lemma}\label{Fano_rho=1_Base_pts}
 Consider a terminal Gorenstein Fano threefold $X$ such that the linear system $|-K_X|$ has a non-empty base locus.
 Any $3$-subgroup $G\subset \Aut(X)$ can be generated by $3$ elements.
\end{lemma}
\begin{proof}
 By \cite[Theorem 0.5]{Gorenstein_Fano} the base locus $B=\mathrm{Bs}(|-K_X|)$ of a Fano threefold 
 is either a point or a projective line $\p^1$. If $B$ is a point, then any automorphism of $X$ fixes it.
 Thus, by Lemma \ref{stabilizer_of_singularity} the group $G$ can be generated by 3 elements. 
 
 If $B$ is isomorphic to $\p^1$, then by Lemma \ref{gl2} the group $G$ fixes a point on $B$. 
 Thus, by Lemma \ref{stabilizer_of_singularity} the group $G$ can be generated by 3 elements.
\end{proof}
	
Recall that for any subvariety $Y$ in a projective space
which is not contained in a hyperplane we have an inequality $\mathrm{codim}(Y)+1\leqslant\deg(Y)$. Moreover,
in case when we have an equality $Y$ is called a \emph{variety of minimal degree}. Such varieties are 
classified, see, for instance, \cite[Theorem 2.2.11]{FANO}; in this situation $Y$ is either a projective space,
a quadric hypersurface, a scroll, a Veronese surfaceor a cone over a scroll or a Veronese surface.
Recall also that a \emph{scroll} is the image of the projectivization of a vector bundle  $\p\left(\bigoplus \mathcal{O}_{\p^1}(a_i)\right)$, 
where all $a_i< 0$ under the map induced by the linear system $|\mathcal{O}_{\p\left(\bigoplus \mathcal{O}_{\p^1}(a_i)\right)}(1)|$.
In the next lemma we study the $3$-groups with a faithful action on a variety of minimal degree.

\begin{lemma}\label{var_of_min_degree}
 Assume that $Y$ is a variety of minimal degree and $G$ is a $3$-subgroup in $\Aut(Y)$.
 If $\dim(Y) = d$ and $d<7$, then $G$ can be generated by $d$ elements.
\end{lemma}
\begin{proof}
 If $Y\cong \p^d$, then by Proposition \ref{gl4} the group $G$ can be generated by $d$ elements.
 
 If $Y$ is a scroll. then the projection from $Y$ to $\p^1$ is canonical; and $G$ preserves the structure 
 of the projective bundle. By Lemma \ref{lemma_Gbundle} and Proposition \ref{gl4} the group $G$ can be 
 generated by $d$ elements.
 
 If $Y$ is a cone over a scroll or a Veronese surface, then the vertex of the cone $Z\subset Y$  is a projective space 
 of codimension $c$ in $Y$. A blow up of $Y$ in $Z$ is a $\p^c$-fiber space over $\p^1$, $\p^2$ or a $(d-c)$-dimensional
 scroll. The group $G$ preserves the vertex of a cone; thus, it acts regularly on the blow up. Then by Lemma \ref{lemma_Gbundle}
 we get the result.
 
 If $Y$ is a $d$-dimensional quadratic hypersurface in $\p^{d+1}$, then by Lemma \ref{lemma_3gr_on_quadric} 
 the group $G$ can be generated by $d$ elements.
\end{proof}

Now we study the Fano varieties such that the base locus of $|-K_X|$ is empty 
and this linear system induces a double cover.

\begin{lemma}\label{Fano_rho=1_2-coverings}
Assume that $X$ is terminal Gorenstein Fano threefold and the linear system 
$|-K_X|$ induces a regular map which is not an embedding. Then a $3$-subgroup $G$
of $\Aut(X)$ can be generated by $3$ elements.
\end{lemma}
\begin{proof}
 By \cite[Proposition 2.1.15]{FANO}, if $\phi_{|-K_X|}$ is base point free, then its degree equals $1$ or $2$.
 Since the map $\phi_{|-K_X|}$ is not an embedding, its image in the linear system $|-K_X|$ is a variety of minimal degree
 with isolated singularities. By Lemmas  \ref{lemma_Gbundle} and \ref{var_of_min_degree}  the group $G$ can be generated by $3$ elements.
\end{proof}
  
\subsubsection{Del Pezzo varieties}
The previous assertions implies the following result for del Pezzo varieties.
\begin{corollary}\label{corollary_3gr_on_dP}
 Assume that $X$ is a terminal Gorenstein $G$-del Pezzo threefold with $\rho(X) = 1$ and~$G$ is a $3$-group.
 Then $G$ can be generated by $4$ elements.
\end{corollary}
\begin{proof}
If the anticanonical system $|-K_X|$ does not induce an embedding to the anticanonical linear system, then
by Lemmas \ref{Fano_rho=1_Base_pts} and \ref{Fano_rho=1_2-coverings} the group $G$ can be generated by $3$ elements.

If the degree of $X$ equals $3$ and $|-K_X|$ induces an embedding, then $G$ is a subgroup of  $\PGL_5(\mathbb{C})$.
Thus by Proposition \ref{gl4} the group $G$ can be generated by $4$ elements.
  
If $X$ is a del Pezzo variety of degree $4$, then by \cite[Theorem 3.2.5]{FANO} it is isomorphic to the complete
intersection of two quadric hypersurfaces. Thus, there is a 1-dimensional pencil of quadrics passing through $X$.
By Lemma \ref{gl2} the group $G$ preserves one hypersurface in this pencil. Thus, we have a faithful action of $G$ on a 
quadric hypersurface of dimension $4$. Lemma  \ref{lemma_3gr_on_quadric} implies that $G$ can be generated by $4$ elements. 

If $X$ is a $G$-del Pezzo threefold of degree $5$, then by \cite[Theorem 1.7]{GFano} it is smooth.
By \cite[Proposition 4.4]{Muk88} we get that the group $\Aut(X)$ is isomorphic to $\PGL_2(\mathbb{C})$. 
Thus, Lemma \ref{gl2} the group~$G$ is cyclic.
\end{proof}

\subsubsection{Fano threefolds of index $1$}
\begin{lemma}\label{lemma_3gr_g_less_5}
Assume that $X$ is a terminal Gorenstein Fano threefold with Picard number $1$, index~$1$ and genus $g<5$.
If $G$ is a $3$-subgroup of $\Aut(X)$ then it can be generated by $4$ elements.
\end{lemma}
\begin{proof}  
If the linear system $|-K_X|$ does not induce an embedding, then by Lemmas \ref{Fano_rho=1_Base_pts} and \ref{Fano_rho=1_2-coverings}
the group $G$ can be generated by $3$ elements. From now on we assume, that $|-K_X|$ induces an embedding to a 
projective space $\p^{g+1}$.
  
If $g=3$ and $|-K_X|$ induces an embedding to $\p^4$, then $G$ is a subgroup of  $\PGL_5(\mathbb{C})$. 
Proposition \ref{gl4} implies that $G$ can be generated by $4$ elements.
   
If $g=4$ then $|-K_X|$ induces an embedding of $X$ to $\p^5$, then by \cite[Proposition 4.1.12]{FANO}
the variety $X$ is an intersection of a cubic and a quadric hypersurfaces. Thus, $G$ preserves a $4$-dimensional
quadric passing through $X$. By Lemma  \ref{lemma_3gr_on_quadric}, the group $G$ can be generated by $4$ elements.
\end{proof}
\begin{lemma}\label{lemma_3gr_g_equals_5}
Assume that $X$ is a terminal Gorenstein Fano threefold with Picard number $1$, index~$1$ and genus $g=5$.
If $G$ is a $3$-subgroup of $\Aut(X)$ then it can be generated by $4$ elements.
\end{lemma}
\begin{proof}
By Lemmas \ref{Fano_rho=1_Base_pts} and \ref{Fano_rho=1_2-coverings} we can assume that $|-K_X|$ 
induces an embedding.

By \cite[Proposition 4.1.12]{FANO} the threefold $X$ is isomorphic to the intersection of $3$ quadric
hypersurfaces $Q_1$, $Q_2$ and $Q_3$ in $\p^6$. The group $G$ acts on the linear system $|-K_X|$. Thus,
the Sylow subgroup of $\widetilde{G}$ the preimage of $G$ in  $\GL_7(\mathbb{C})$ has the representation 
in the vector space $H^0(X, \mathcal{O}(K_X)))^{\vee}$ of dimension 7. There exists a $4$-dimensional 
subrepresentation of $\widetilde{G}$ in $H^0(X, \mathcal{O}(K_X)))^{\vee}$. Therefore, we can find a 
$G$-invariant subspace $\Pi \cong \p^3$ of dimension $3$.

Consider the intersection $Y = \Pi\subset X$. Since $X$ is irreducible, it is not empty and proper in $\Pi$.
Moreover, $Y$ is $G$-invariant subvariety of $X$. By Lemma \ref{lemma_intersection_3_quadrics} there is a subgroup $H$
of $G$ of index less than or equal to $3$ preserving a point $y$ on $Y$. Then Lemmas \ref{stabilizer_of_singularity}
and \ref{subgroup_of_small_index} implies that  $G$ can be generated by $4$ elements.  
\end{proof}
  
\begin{lemma}\label{lemma_3gr_g_equals_6}
Assume that $X$ is a terminal Gorenstein Fano threefold with Picard number $1$, index~$1$ and genus $g=6$.
If $G$ is a $3$-subgroup of $\Aut(X)$ then it can be generated by $4$ elements.
\end{lemma}
\begin{proof}
By Lemmas \ref{Fano_rho=1_Base_pts} and \ref{Fano_rho=1_2-coverings} we can assume that $|-K_X|$ 
induces an embedding.

Denote by $Y$ the intersection of all quadric hypersurfaces in the linear system $|-K_X|$ passing through $X$ (c.f. \cite[Example 4.3.3]{X10}).
It is a $G$-invariant variety; thus, $G\subset\Aut(Y)$.
If $Y$ does not coinside with $X$, then it is of dimension $4$ and it is a variety of minimal degree 
(see, for instance, proof of \cite[Proposition 2.3]{Isk_78}). Therefore, by Lemma \ref{var_of_min_degree}
the group $G$ can be generated by $4$ elements.

Now assume that $X$ coinsides with the intersection of all quadratic hypersurfaces passing through it.
Any curve in the intersection of $X$ and two hyperplane section of $X$ is an intersection of quadric hypersurfaces too.
Its genus equals to 6.
By \cite[Proposition 2.12]{Gushel_map} this is a  Clifford general curve. By  \cite[Proposition 2.15]{Gushel_map}
the threefold $X$ is a Gushel--Mukai variety \cite[Definition 2.1]{Gushel_map}. Thus, there exists an 
embedding or a double cover~\mbox{$X \to \Gr(2,5)$} and this map is $\Aut(X)$ invariant. Then by
Lemma \ref{lemma_Gbundle} the group $G$ is a subgroup of $\Aut(\Gr(2,5))$. By Propositions \ref{GR}  and \ref{gl4}
the group $G$ can be generated by $4$ elements.
\end{proof}

\begin{lemma}[{c.f. \cite[Lemma 7.6]{p-subgroups}}]\label{lemma_qfact}
Assume that $X$ is a terminal Gorenstein $G\Q$-Fano threefold with Picard number $1$, index~$1$ and genus $g=8,\ 9$ or $12$.
If $G$ is a $3$-group then $X$ is a $\mathbb{Q}$-factorial variety.
\end{lemma}
\begin{proof}
As in proof of  \cite[Lemma 7.6]{p-subgroups} we can show that if $X$ is not a $\mathbb{Q}$-factorial variety,
then it is birational to a terminal Gorenstein Fano threefold $Y$ with the following condition
on the degree of the canonical class:
  \begin{equation*}
   -K_Y^3 \geqslant 2g(X) - 2 + (3-1)(4g(X) - 6) \geqslant 66.
  \end{equation*}
By \cite[Tables 12.6--12.7]{FANO} the degree of any smooth Fano variety less than 65. Since the degrees of the canonical 
classes are same in all fibers in the Namikawa smoothing \cite{Fano_smoothing} we get a contraction. Thus, $X$ is 
a $\mathbb{Q}$-factorial variety.
\end{proof}

Now we study singular Fano varieties of genus greater than or equal to $8$.
\begin{lemma}\label{lemma_r1_g8912}
Assume that $X$ is a terminal Gorenstein $G\Q$-Fano threefold with Picard number $1$, index~$1$ and genus $g=8,\ 9$ or $12$.
If $G$ is a $3$-group then it can be generated by $4$ elements.
\end{lemma}
\begin{proof}
By Lemma \ref{lemma_qfact} the variety $X$ is $\mathbb{Q}$-factorial. By \cite[Theorem 1]{sing_X8912} there are 
at least 1 and at most $5$ singular points on $X$. Thus, Lemma \ref{lemma_less_27_singularities} implies the result.
\end{proof}

Now we study smooth Fano threefolds of genus greater than or equal to $7$. In order to do 
consider a smooth Fano threefold and its very ample linear system $|-K_X|$. Denote by $S(X)$ the Hilbert scheme of conics on $X$ 
in the anticanonical embedding. We recall the following set of facts about these Hilbert 
schemes\mbox{\cite[Theorem 1.1.1, Lemma 4.2.1, Lemma 4.3.4, Corolla\-ry~4.3.5]{Hilb_Fano}}.
\begin{proposition}\label{Conic_on_Fano} 
   If $X$ is a smooth Fano threefold of index $1$ and $g(X)\geqslant 6$, then~$S(X)$ is a smooth surface and
   \begin{enumerate}
    \item[(i)] if $g=7$, then $S(X)$ is a symmetric square of a smooth curve $C$ of genus~$7$;
    \item[(ii)] if $g = 9$, then $S(X)$ is a projectivization of a vector bundle of rank $2$ over a curve of genus $3$;
    \item[(iii)] if $g = 10$, then $S(X)$ --- is a Jacobian $\mathrm{Jac}(C)$  of a smooth curve $C$ of genus $2$ ;
    \item[(iv)] if $g = 12$, then $S(X)\cong \p^2$.
   \end{enumerate}
   Moreover, the group $\Aut(X)$ acts faithfully on $S(X)$ and in case of $g=7$ on the curve $C$.
\end{proposition}
This assertion implies the bound for smooth Fano threefolds of big genus.
\begin{lemma}\label{lemma_Fano_r1_g7910}
Assume that $X$ is a smooth Fano threefold of index $1$ and $g(X)=7,\ 9, \ 10$ or $12$.
Then a $3$-subgroup~$G$ in $\Aut(X)$ can be generated by $4$ elements.
\end{lemma}
\begin{proof}
By Proposition \ref{Conic_on_Fano} the group $G$ is a subgroup of $\Aut(S(X))$. 

If $g=12$, the group  $\Aut(S(X))$ is isomorphic to $\PGL_3(\mathbb{C})$. Thus, Proposition \ref{gl4} implies 
that $G$ can be generated by $2$ elements.

If $g=10$, the surface $S(X)$ is a Jacobian of a curve of genus $2$ and the group $G$ preserves its polarization.
Thus, by Lemma \ref{lemma_jac} the group $G$ can be generated by $2$ elements.

If $g=9$, by Lemma \ref{lemma_proj_over_curve} the group $G$ can be generated by $4$ elements.

If $g=7$, by Lemma \ref{Conic_on_Fano} the group $G$ is a subgroup of $\Aut(C)$ and $C$ is a curve of genus $7$.
The bound \eqref{eq_estimation_on_curve} implies
\begin{equation*}
 |G|\leqslant 16\cdot (g(C)-1) = 96<3^5.
\end{equation*}
Thus, the group $G$ can be generated by $4$ elements.
\end{proof}
The next assertion describes 3-groups with faithful action on Fano threefolds of genus 8.
\begin{lemma}\label{lemma_Fano_r1_g8}
Assume that $X$ is a smooth Fano threefold of index $1$ and genus $8$.
Then a $3$-subgroup~$G$ in $\Aut(X)$ can be generated by $4$ elements.
\end{lemma}
\begin{proof}  
As was shown in \cite[Section B.6]{Hilb_Fano} there exists a canonical construction of a smooth del Pezzo threefold $Y$ of degree $3$
by a a smooth Fano threefold $X$ of index $1$ and genus $8$. The action of~$G$ on $X$ induces an action of $G$ on $Y$; this action 
can be not faithful. By \cite[Proposition B.6.3]{Hilb_Fano} the Hilbert scheme $S(Y)$ of lines on $Y$ is isomorphic to the Hilbert scheme  
of conics~$\Sigma(X)$ on $X$. The action of $G$ on $\Sigma(X)$ is faithful by Proposition \ref{Conic_on_Fano}. Thus, the action of $G$
on $Y$ is also faithful. Thus, by Corollary \ref{corollary_3gr_on_dP} the group $G$ can be generated by $4$ elements.
\end{proof}  

\subsection{Proof of Theorem \ref{main_theorem_dim3}}
 
Remark \ref{rmk_proof_main_thm} proves the first assertion of Theorem  \ref{main_theorem_dim3}. 
To prove the second assertion we consider several cases.

If $X_0$ is a $G$-Mori fiber space, then by Lemma \ref{lemma_mb} the group $G$ can be generated by $4$ elements.
From now on we assume that $X_0$ is a terminal $G$-Fano threefold.

If on $X_0$ there is a non-Gorenstein terminal singularity and it is not the variety described in point~$(a)$,
then by Corollary \ref{corollary_reid_imply} the group $G$ can be generated by $4$ elements. From now on we assume
that $X_0$ has only Gorenstein singularities.

If $X_0$ is a Fano threefold and $\rho(X)>1$, then by Lemma \ref{Fano_big_rho}  the group $G$ can be generated by $4$ elements.

If $X_0$ is a Fano threefold, $\rho(X) = 1$ and index of $X$ is greater than 2, then $X_0$ is a quadric hyper surface or a projective space.
By Proposition \ref{gl4} and \ref{var_of_min_degree} the group $G$ can be generated by $3$ elements.

If $X_0$ is a del Pezzo variety, then by Corollary \ref{corollary_3gr_on_dP} the group $G$ can be generated by $4$ elements.

If $X_0$ is a Fano threefold of index $1$ and genus less than $7$,  then  by Lemmas ~\mbox{\ref{lemma_3gr_g_less_5}, \ref{lemma_3gr_g_equals_5}}
 and \ref{lemma_3gr_g_equals_6} the group $G$ can be generated by $4$ elements.
 
If $X_0$ is a Gorenstein singular Fano threefold of index $1$ and genera 8, 9 or 12, then by Lemma~\ref{lemma_r1_g8912}
the group $G$ can be generated by $4$ elements.
 
If $X_0$ is smooth Fano threefold of index $1$ and genus greater than 6, then by Lemmas \ref{lemma_Fano_r1_g7910}
and \ref{lemma_Fano_r1_g8} the group $G$ can be generated by $4$ elements.
 
If $X_0$ is a singular Gorenstien Fano threefold of index $1$ and genera 7 and 10 and the number of singularities 
on $X_0$ differs from 9 or 18, then by Corollary \ref{corollary_namikawa_imply} the group $G$ can be generated by $4$ elements.
This finishes the proof of  Theorem \ref{main_theorem_dim3}.

\begin{remark}\label{rmk_final}
Theorem \ref{main_theorem_dim3} implies that if there exists a $3$-subgroup $G$ of a group $\Bir(X)$ of birational automorphisms of
a rationally connected threefold $X$ and $G$ can not be generated by less than $5$ elements, then the group $G$ acts faithfully on threefolds
with very specific properties. Indeed, by Proposition \ref{regularization}
there exists a regularization $\widetilde{X}$ of the action of $G$ on $X$. Applying a $G$-equivariant minimal model program to $\widetilde{X}$
we get a terminal $G$-Mori fiber space with regular and faithful action of $G$. Since $G$ can not be generated by less than $5$ elements, $X_0$
is a threefold described in points  $(a)$ or $(b)$ of theorem \ref{main_theorem_dim3}. 

Moreover, this argument gives another proof of the first point of Theorem \ref{main_theorem_dim3} since by Corollaries~\ref{corollary_namikawa_imply}
and~\ref{corollary_reid_imply} any $3$-group with a faithful action on threefolds described in points  $(a)$ or $(b)$ can be generated by $5$ elements.
\end{remark}

\appendix
\section{{3-groups and their representations}}\label{appendix}
\subsection{Generators}
In this section we estimate number of elements generating groups. Consider a situation where we have an estimation 
of the number of generators of a subgroup of a small index of our group. In such situation we can estimate the number
of generators of the whole group using the next assertion.
\begin{lemma}\label{subgroup_of_small_index}
Assume that $G$ is a $p$-group and $H$ is its subgroup of index $p^n$. 
If $H$ can be generated by $m$ elements, then $G$ can be generated by $n+m$ elements.
\end{lemma}
\begin{proof}
Choose an element $g\in G$ outside $H$ and consider a subgroup $H'\subset G$ generated by $m$ generators of $H$ and $g$.
We have
\begin{equation*}
 [G:H']<[G:H].
\end{equation*}    
Thus, $H'$ is a subgroup of index $p^{n-1}$ or less. Then by induction we show that $G$ can be generated by $n+m$ elements.
\end{proof}
Now consider the case when the group $G$ is an extension of two groups and we can estimate numbers of generators of them.
\begin{lemma}\label{generators_of_group_extension}
Assume that a finite group $G$ is an extension of a group $G_2$ by a group $G_1$ 
and each group~$G_i$ can be generated by $n_i$ elements for $i=1$ and $2$. 
Then the group $G$ can be generated by~\mbox{$n_1+n_2$} elements.
\end{lemma}
\begin{proof}
Choose sets of generators $x_1,\dots,x_{n_1}$ and $y_1,\dots,y_{n_2}$ of groups $G_1$ and $G_2$.
Denote by $\widetilde{x_i}$ the image of $x_i$ in $G$ and choose a preimage $\widetilde{y_j}$ of $y_j$ in $G$.

Take an element $z$ in $G$. Then the image of $z$ in $G_2$ equals to a product of generators $\prod y_i^{d_i}$.
Then the product $z\cdot(\prod \widetilde{y_i}^{d_i})^{-1}$ maps to unity in $G_2$. Thus, this product equals 
to~$\prod \widetilde{x_j}^{c_j}$. Therefore, $z$ can be represented as a product of $x_j$ and $y_i$ and these elements generate
the group $G$.
\end{proof}
If $H$ is a subgroup of $G$ and we know the number of generators of $G$ then we can conclude almost no information about the
number of generators of $H$. For example, any finite group can be embedded to a group of permutation $\SG_N$; though, $\SG_N$
can be generated by a transposition and a cycle of length~$N$. Nevertheless, if $H$ is a subgroup of a direct product we can
extract some information.
\begin{lemma}\label{subgroup_of_direct_product}
Assume that $H\subset \prod\limits_{i=1}^N G_i$ and for each $i$ any subgroup of $G_i$ can be generated 
by $n_i$ elements. Then $H$ can be generated by $\sum\limits_{i=1}^N n_i$ elements.
\end{lemma}
\begin{proof}
We prove this by induction by $N$. If $N=1$, then the assertion is trivial. Assume that the assertion is true for $N-1$ and show it for $N$.
Denote by~$\mathrm{pr}_N$ the projection form the product to $G_N$: 
\begin{equation*}
 \mathrm{pr}_N: \prod\limits_{i=1}^N G_i\to G_N.
\end{equation*}
The groups $\mathrm{pr}_N(H)\subset G_N$ and $\mathrm{Ker}(\mathrm{pr}_N)\cap H \prod\limits_{i=1}^{N-1}G_i$
can be generated by $n_N$ and $\sum\limits_{i=1}^{N-1}n_i$ respectively. The group $H$ is the following extension:
\begin{equation*}
 1\to \mathrm{Ker}(\mathrm{pr}_N)\cap H\to H \to \mathrm{pr}_N(H) \to 1
\end{equation*}
By Lemma \ref{generators_of_group_extension} we conclude that $H$ can be generated by $\sum\limits_{i=1}^N n_i$ elements.
\end{proof}

\subsection{Heisenberg group}
Denote by $\H_3$ a group of $27$ elements with three generators~$x$,~$y$ and~$z$ and relations
\begin{align*}
 &x^3 = y^3 = z^3 = 1;\\
 &xyx^{-1}y^{-1} = z.
\end{align*}
This group is called \emph{Heisenberg group}, it can be generated by two elements $x$ and $y$ and its center
is generated by $z$. Thus, we have the following exact sequence.
\begin{equation*}
 1 \to \mathrm{C}_3\langle z \rangle \to \H_3 \to \mathrm{C}_3\langle x, y\rangle \to 1.
\end{equation*}
Here the group $\mathrm{C}_3\langle x, y\rangle$ is isomorphic to a direct product of cyclic groups $\mathrm{C}_3\times \mathrm{C}_3$.
The group $\H_3$ can be described as a subgroup of $\SL_3(\mathbb{C})$ generated by the following matrices
\begin{equation*}
 X = \left(\begin{smallmatrix}
       0&0&1\\ 1&0&0\\ 0&1&0
     \end{smallmatrix}\right)
\text{ and }
Y = \left(\begin{smallmatrix}
     1&0&0\\
     0& \zeta & 0\\
     0&0& \zeta^2
    \end{smallmatrix}\right).
\end{equation*}
Here $\zeta$ is a primitive root of unity. In this realization the center of the group $\H_3$ is inside the group of scalar 
matrices. Denote by $N(\H_3)$ the normalizer of a Heisenberg group in $\SL_3(\mathbb{C})$. 
\begin{lemma}[{\cite[\S 8.5]{Normalizer_of_H3}}]
 The quotient group $N(\H_3)/\H_3$ is isomorphic to $\SL_2(\mathbb{F}_3)$.
\end{lemma}
The maximal $3$-subgroup of $\SL_2(\mathbb{F}_3)$ is isomorphic to a cyclic group $\mathrm{C}_3$. 
Denote by~$\widetilde{\H_3}$ the maximal~\mbox{$3$-sub}\-group of~$N(\H_3)$. Then it is the following extension.
\begin{equation*}
   1 \to \H_3 \to \widetilde{\H_3} \to \mathrm{C}_3 \to 1
\end{equation*}
\begin{lemma}\label{lemma_image_of_H}
 The image of the group $\widetilde{\H}_3\subset \SL_3(\mathbb{C})$ under the projection to the group $\PGL_3(\mathbb{C})$ 
 is isomorphic to $\H_3$.
\end{lemma}
\begin{proof}
 Denote the standard projection $\SL_3(\mathbb{C})\to \PGL_3(\mathbb{C})$ by $\pi$.
 In the center of the group $\H_3$ in~$\SL_3(\mathbb{C})$ all matrices are scalar; and they map to unity in $\PGL_3(\mathbb{C})$.
 Thus, the group $W=\pi(\H_3)$ can be considered as a vector space of dimension $2$ over $\mathbb{F}_3$. 
 Denote by $m$ the generator of the Sylow~\mbox{$3$-sub}\-group of $\SL_2(\mathbb{F}_3)$. 
 In some coordinates $w_1$ and $w_2$ of $W$ the generator $m$ acts as the following matrix:
 \begin{equation*}
  M = \left(\begin{smallmatrix}
            1&1 \\0& 1
           \end{smallmatrix}
       \right).
 \end{equation*}
Thus,  the group $\pi(\widetilde{\H_3})$ is a non-abelian group of order $27$ and all its elements are of order $3$. Therefore, it is a Heisenberg group.
\end{proof}

\subsection{Consequents of linear algebra}
 
We start with a description of $p$-subgroups of $\GL_2(\mathbb{C})$. 
\begin{lemma}\label{gl2}
If $p\ne 2$ is a prime number, then any $p$-subgroup $G$ in $\GL_2(\mathbb{C})$ is abelian and can be generated by $2$ elements.
 \begin{equation*}
  G\cong \mathrm{C}_{p^n} \times \mathrm{C}_{p^m}.
 \end{equation*}
Moreover, any $p$-subgroup of $\PGL_2(\mathbb{C})$ is cyclic and its action on a projective line preserves a point.
\end{lemma}
\begin{proof}
Any embedding of a finite group $G$ to $\GL_2(\mathbb{C})$ induces a faithful representation of dimension $2$ of the group $G$.
The order of the group $G$ divides the dimension of any irreducible representation of~$G$. Thus, if $p\ne 2$, a representation of 
dimension 2 of $G$ is a sum of linear representations. Therefore, we get the result. 
\end{proof}
Now we consider Lie groups of higher rank.
\begin{lemma}\label{lemma_gl3}
 If $G$ is a $3$-subgroup of $\PGL_3(\mathbb{C})$, then it can be generated by $2$ elements.
 If $G$ is a~\mbox{$3$-sub}\-group of $\GL_3(\mathbb{C})$, then it can be generated by $3$ elements.
\end{lemma}
\begin{proof}
The first assertion can be proved as in \cite[Section 6.4]{Borel}. 
To estimate the number of generators of a $3$-subgroup $G$ of $\GL_3(\mathbb{C})$ we consider the following
exact sequence.
 \begin{equation*}
  1 \to \mathbb{C}^* \to  \GL_3(\mathbb{C}) \xrightarrow{\pi} \PGL_3(\mathbb{C}) \to 1.
 \end{equation*}
The kernel of the map $\pi|_{G}$ is a subgroup of $\mathbb{C}^*$; thus, it is a cyclic group.
The image~$\pi(G)$ is a~\mbox{$3$-sub}\-group of $\PGL_3(\mathbb{C})$, thus, it can be generated by  $2$ elements.
Lemma \ref{generators_of_group_extension} implies the result.
\end{proof}

The next assertion describes $p$-subgroups of a compact Lie group for all prime numbers $p$.
\begin{theorem}[{\cite[Theorem 1]{BS}, \cite[Section 6.4]{Borel}}]\label{thm_psgroup_of_Lie}
Assume that $\mathscr{G}$ is a compact Lie group and~$G$ is a~\mbox{$p$-sub}\-group in $\mathscr{G}$.
Then there exists a torus $T\subset\mathscr{G}$ such that the group $G$ is a subgroup of a normalizer $N(T)$. 

In particular, if $\mathscr{G}$ is $\SL_p(\mathbb{C})$ or $\PGL_p(\mathbb{C})$, then
the quotient group $G/(G\cap T)$ is contained in~$\mathbb{Z}/p\mathbb{Z}$, generated by matrices of permutations.
\end{theorem}
\begin{corollary}\label{corollary_3sgr_gl3_3generators}
Assume that a $3$-subgroup $G$ of $\GL_3(\mathbb{C})$ can not be generated by $3$ elements.
Then either $G$ contains an abelian group $\mathrm{C_3^3}$, or in some basis it contains 
the following matrices $($here $\zeta$ is a primitive cube root of unity and $\lambda_i$ are roots of unity of degree $3^{k_i})$.
 \begin{align}\label{eq_matrix_in_3group}
 \sigma = \begin{pmatrix}
   0 & \lambda_2 & 0 \\ 0& 0&\lambda_3 \\ \lambda_1 & 0&0  
  \end{pmatrix}
 &&
  t=\begin{pmatrix}
   1&0&0\\ 0&\zeta&0 \\0&0&\zeta^2
  \end{pmatrix}
 \end{align}	
\end{corollary}
\begin{proof}
 The group $G\cap T$ is abelian. If its rank equals $3$, then $G$ contains $\mathrm{C_3^3}$. 
 Since $G$ can not be generated by 2 elements, $\mathrm{rk}(G\cap T)$ is greater than or equal to $2$.
 Assume, that it equals $2$.
 
 By Theorem \ref{thm_psgroup_of_Lie} in some basis the group $G$ is a subgroup of a normalizer of 
 the torus  $T\subset\GL_3(\mathbb{C})$ of diagonal matrices. Thus, there is an element $\sigma\in G$ such that
 the conjugation by $\sigma$ acts on $G\cap T$ as a permutation of coordinates.
 
 Consider $t$ a non trivial element of $G\cap T$. If elements $t$ and $\sigma t\sigma^{-1}$ are not 
 proportional, elements~$\sigma$ and~$t$ generate an abelian subgroup of rank $3$ in $G\cap T$ 
 and this contradicts to our assumption. Thus, element $\sigma t\sigma^{-1}$ is proportional to $t$; and $G$
 contains a scalar matrix or the second matrix in \eqref{eq_matrix_in_3group}.
 Since the rank of $G\cap T$ equals $2$, they both are in $G$.
\end{proof}
Now we can estimate the number of generators of a $3$-subgroup in $\GL_n(\mathbb{C})$ for $n$ greater than $3$.
\begin{proposition}\label{gl4}
 Assume that $3<n<9$ and $G$ is a $3$-subgroup in $\GL_n(\mathbb{C})$. Then $G$ can be generated by $n$ elements.
 If $G$ is a $3$-subgroup in $\PGL_n(\mathbb{C})$, then $G$ can be generated by $n-1$ elements.
 \end{proposition}
\begin{proof}
The embedding of the group $G$ to $\GL_n(\mathbb{C})$ induces a representation of $G$ of dimension $n$.
Since by assumption $n<9$ this representation is isomorphic to a direct sum of irreducible representations of dimensions 1 and 3.
Thus, we have 
\begin{equation}\label{eq_decomp_subgroup_gl}
G\subset \Pi_{i=1}^{n-3k}G'_i \times \Pi_{j=0}^kG''_j	.
\end{equation}
Here $G'_i$ are cyclic 3-groups and $G''j$ are $3$-subgroups of $\GL_3(\mathbb{C})$.
By Lemmas \ref{subgroup_of_direct_product} and \ref{lemma_gl3} the group~$G$ can be generated by $n$ elements.

Now consider a $3$-subgroup $G$ of  $\PGL_n(\mathbb{C})$. Denote by $\widetilde{G}$ a $3$-subgroup of $\GL_n(\mathbb{C})$ 
which maps surjectively to $G$ under the canonical projection. If $\widetilde{G}$ can be generated by $n-1$ elements, then $G$ also can.
Assume, that $\widetilde{G}$ can not be generated by less than $n$ elements. 

By Corollary \ref{corollary_3sgr_gl3_3generators} the group $G$ either contains an abelian subgroup of rank $n$ 
or
\begin{equation}\label{eq_1}
\widetilde{G} = H'\times H'', 
\end{equation}
where $H'\subset \GL_3(\mathbb{C})$ and~\mbox{$H'' \subset \GL_{n-3}(\mathbb{C})$}.
Moreover, $H'$ contains matrices from \eqref{eq_matrix_in_3group} and the rank of $H\cap T$ is greater than or equal to $2$.

Assume, that $\widetilde{G}$ is abelian of rank $n$. Then its image in $\PGL_n(\mathbb{C})$ can be generated by $n-1$ element.
If $\widetilde{G}$ is not abelian, then we have a decomposition \eqref{eq_1}. The group $H'$ can not be generated by less than 3 elements
only if $\mathrm{rk}(H'\cap T) = 2$ and $H'$ is generated by a scalar matrix and matrices \eqref{eq_matrix_in_3group}. 
Thus, image of $H'$ in $\PGL_3(\mathbb{C})$ can be generated by $2$ elements. The image of $H''$ in $\PGL_{n-3}(\mathbb{C})$ can
be generated by $n-3$ elements. Therefore, we get the result.
\end{proof}
\begin{corollary}\label{corollary_ab_subgroup_in_3gr}
 If $G\subset \PGL_n(\mathbb{C})$ is a $3$-subgroup which can not be generated by less than $k$ elements and $n<9$, 
 then $G$ contains an abelian subgroup $A$ of rank $k$.
\end{corollary}
\begin{proof}
 This assertion is true for images of subgroups of diagonal torus in $\GL_n(\mathbb{C})$ as well as for images of groups
 generated by matrices \eqref{eq_matrix_in_3group} and scalar matrices. Therefore, it is true for the image of any~$3$-subgroup 
 in  $\GL_n(\mathbb{C})$.
\end{proof}

\begin{lemma}\label{3-subgroups_in_gl3Z}
 Any $3$-subgroup in the group $\GL_2(\mathbb{Z})$ or $\GL_3(\mathbb{Z})$ is cyclic.
\end{lemma}
This is proved in \cite{finite_subgroups_in_gl3Z}; so we skip the proof despite it follows from the above assertions.
\begin{lemma}\label{3-group_acts_on_Z^2}
 Consider a $3$-group $G$ with an action on the lattice $\Lambda\cong\mathbb{Z}^2$. 
 Then the invariant lattice~$\Lambda^G$ can not be of rank $1$.
\end{lemma}
\begin{proof}
 The group $G$ maps to the group of automorphisms of $\Lambda$ which is isomorphic to $\SL_2(\mathbb{Z})$. 
 By Lemma \ref{3-subgroups_in_gl3Z} the image of group $G$ is cyclic and it is isomorphic to $\mathrm{C}_{3^n}$.
 Denote the generator of this group by $\gamma$. The action of the operator $\gamma$ on $\Lambda$ has two eigenvalues $\lambda_1$ and $\lambda_2$
 If the invariant sublattice $\Lambda^G$ is non-trivial, then one eigenvalue $\lambda_1$ equals 1. Since $\lambda_1\cdot\lambda_2=1$, we get that
 $\lambda_2$ also equals 1. Thus, either $\Lambda^G$ equals 0 or $\Lambda$.
\end{proof}

\subsection{Quadric hypersurfaces}
In this section we study quadric hypersurfaces and 3-groups with actions on them.
\begin{lemma}\label{lemma_3gr_on_quadric}
 Assume that $X$ is a quadric hypersurface in $\p^{d+1}$ and $d<7$
 If $G$ is a $3$-subgroup in $\Aut(X)$, then $G$ can be generated by $d$ elements.
\end{lemma}
\begin{proof}
 If $X$ is a smooth hypersurface, then $G$ is a subgroup in $\PSO_{d+2}(\mathbb{C})$. 
 By Corollary \ref{corollary_ab_subgroup_in_3gr} if $G$ can not be generated by less than $m$
 elements, then there is an abelian subgroup $A\subset G$ of rank $m$. 
 
 However, any abelian subgroup of $\PSO_{d+2}(\mathbb{C})$ is a subgroup of the torus of this group 
 and the dimension of torus of $\PSO_{d+2}(\mathbb{C})$ equals~$\lfloor \frac{d+2}{2} \rfloor$.
 Thus, we get the result:
 \begin{equation*}
  m\leqslant \left\lfloor \frac{d+2}{2} \right\rfloor \leqslant d.
 \end{equation*}
 
 If $X$ is a singular quadratic hypersurface, it is a cone with a vertex $Z$ over a smooth quadric hypersurface.
 After a blow up of this vertex we get a $G$-invariant projective bundle over a smooth quadric hypersurface with a 
 faithful action of $G$. Lemma  \ref{lemma_Gbundle} and Proposition~\ref{gl4} implies the result.
\end{proof}
Now we estimate the number of generators of a $3$-group with an action on the intersection of several quadric surfaces.
\begin{lemma}\label{lemma_intersection_3_quadrics}
 Consider a non-empty proper intersection $Y$ of $3$ quadric surfaces in $\p^3$.
 Assume that $G$ is a $3$-group with an action on $Y$. Then there exists a subgroup $H$ in $G$ of index less than of equal to $3$ 
 such that $H$ preserves a point $y$ in $Y$.
\end{lemma}
\begin{proof}
 The dimension of $Y$ can be equal to 0, 1 or 2. If $\dim(Y) = 0$, then~$Y$ is a union of less than or equal to 8 points.
 Since the length of the orbit of a $3$-group is a power of $3$, then for any point $y\in Y$ the stabilizer $\mathrm{Stab}(y) \subset G$
 is a subgroup of index $3$ or less.
 
 If $\dim(Y) = 1$, then $Y$ is a union of curves $C_1\cup\dots\cup C_k$ in $\p^3$ and 
 \begin{equation*}
  \deg(C_1)+ \dots + \deg(C_k) = 4.
 \end{equation*}
 If $Y$ is singular, then there are less than or equal to $6$ singular points on it. The group $G$ preserves singular locus of $C$;
 thus the stabilizer of a singular point is s subgroup of index 3 or less.
 
 If $Y$ is a smooth irreducible curve, it is an intersection of two general quadric surfaces. Thus, the family of quadric surfaces 
 passing through $Y$ is of dimension $1$. There are $4$ surfaces in this family which are singular. The group $G$ fixes the family 
 and the set of singular surfaces in it. Consider the subgroup $H$ in $G$ which preserves 4 degenerate fibers in the family. The index 
 of $H$ is less than or equal to $3$ and $H$ fixes the family pointwise. Therefore, $H$ acts trivially on $Y$.
 
 If $\dim(Y)=2$, then $Y$ is a quadric surface in $\p^3$. If $Y$ is singular, then it is a cone and its vertex is either a point
 or a line. In any case by Proposition \ref{gl2} the group $G$ fixes a point on $Y$. If $Y$ is smooth, then $Y\cong \p^1\times\p^1$.
 By Proposition \ref{gl2} there exists a $G$-fixed point on $Y$.
\end{proof}

\bibliographystyle{alpha}
\bibliography{cremona}

\end{document}